\documentclass[11pt,a4paper]{amsart}
\usepackage[utf8]{inputenc}
\usepackage{amssymb}
\usepackage{amsmath}
\usepackage{amsthm}
\usepackage[a4paper,margin=1in]{geometry}
\numberwithin{equation}{section}
\usepackage{hyperref}
\usepackage{pbox}
\usepackage{graphicx}
\usepackage{amsrefs}

\DefineSimpleKey{bib}{primaryclass}{}
\DefineSimpleKey{bib}{archiveprefix}{}

\BibSpec{arXiv}{%
  +{}{\PrintAuthors}{author}
  +{,}{ \textit}{title}
  +{}{ \parenthesize}{date}
  +{,}{ arXiv }{eprint}
  +{,}{ primary class }{primaryclass}
}

\usepackage{caption}
\usepackage{subcaption}
\usepackage{float}
\usepackage{pax}
\usepackage{pdfpages}
\usepackage{enumerate}
\usepackage[shortlabels]{enumitem}
\usepackage{stmaryrd}

\usepackage{tikz}
\usepackage{tikz-cd}
\usepackage{mathtools}
\usepackage{comment}
\usepackage{blkarray}

\theoremstyle{plain}
\newtheorem{theorem}{Theorem}[section]

\theoremstyle{definition}
\newtheorem{remark}[theorem]{Remark}

\newtheorem{definition}[theorem]{Definition}
\newtheorem{thrmdefinition}[theorem]{Theorem-Definition}

\newtheorem{exmp}[theorem]{Example}

\theoremstyle{plain}
\newtheorem{proposition}[theorem]{Proposition}

\theoremstyle{plain}
\newtheorem{lemma}[theorem]{Lemma}

\theoremstyle{plain}

\DeclareMathOperator{\reg}{reg}

\DeclareMathOperator{\dimn}{dim}

\DeclareMathOperator{\hilb}{Hilb}

\DeclareMathOperator{\pic}{Pic}
\DeclareMathOperator{\gr}{Gr}
\DeclareMathOperator{\tors}{tors}
\DeclareMathOperator{\sym}{Sym}
\DeclareMathOperator{\cox}{Cox}
\DeclareMathOperator{\nef}{Nef}

\DeclareMathOperator{\eff}{\overline{Eff}}

\title{Gotzmann's persistence theorem for Mori dream spaces}
\author{Patience Ablett}
\address{Warwick Mathematics Institute, University of Warwick, Coventry, CV4 7EZ}
\email{patience.ablett@warwick.ac.uk}

\usepackage{calc}
\begin{document}

\maketitle
\begin{abstract}
    Gotzmann’s persistence theorem provides a method for determining the Hilbert polynomial of a subscheme of projective space by evaluating the Hilbert function at only two points, irrespective of the dimension of the ambient space. In \cite{ablett2024gotzmann} we established an analogue of Gotzmann's persistence theorem for smooth projective toric varieties. We generalise our results to the setting of Mori dream spaces, whose associated Cox rings are also finitely generated. We also give an alternative, stronger, persistence result for points in products of projective spaces.
\end{abstract}

\section{Introduction}
The Hilbert scheme $\hilb_{P(t)}(\mathbb{P}^n)$ parameterises subschemes of $\mathbb{P}^n$ with Hilbert polynomial $P(t)$, or equivalently saturated homogeneous ideals in $k[x_0,\dots,x_n]$ with Hilbert polynomial $P(t)$. From an algebraic perspective this scheme is an example of a multigraded Hilbert scheme, as in \cite{haiman2002multigraded}. In fact, for any smooth projective toric variety $X$ with Cox ring $S$, the Hilbert function of a homogeneous ideal of $S$ eventually agrees with a polynomial $P(\boldsymbol{t})$. We can therefore define the associated multigraded Hilbert scheme $\hilb_{P(\boldsymbol{t})}(X)$, parameterising saturated homogeneous ideals with Hilbert polynomial $P(\boldsymbol{t})$.

Gotzmann's regularity and persistence theorems \cite{gotzmann1978bedingung} are useful tools for finding explicit equations for $\hilb_P(\mathbb{P}^n)$. A generalisation of these results would enable us to verify that a set of points is a supportive set, as in \cite{haiman2002multigraded}*{Section 3}, which is a key step in the construction of $\hilb_{P(\boldsymbol{t})}(X)$. Maclagan and Smith \cite{maclagan2003uniform} give a generalisation of Gotzmann's regularity theorem for smooth projective toric varieties. In Ablett \cite{ablett2024gotzmann} we generalised Gotzmann's persistence theorem to this setting. Explicitly, for an ideal in the Cox ring of a Picard-rank-s smooth projective toric variety we find a set of $2^s$ points such that checking the Hilbert function of an ideal in these points verifies its Hilbert polynomial.

However, smooth projective toric varieties are not the only varieties with a finitely generated Cox ring. In fact, the property of having a finitely generated Cox ring characterises a larger class of varieties: Mori dream spaces. Note that given an ideal sheaf $\mathcal{J}$ on a Mori dream space $X$, we can associate a homogeneous ideal in the Cox ring $S$, generated by sections vanishing on the corresponding closed subscheme. Conversely, every ideal sheaf on $X$ arises as the sheaf associated to a homogeneous ideal of $S$, \cite{MR3307753}*{Proposition 4.2.1.11}. The Hilbert function of a homogeneous ideal $J \subseteq S$ therefore agrees with a polynomial sufficiently far into the nef cone, given by the Euler characteristic of the twisted structure sheaf of the corresponding subscheme (see Proposition \ref{polyex}). It is therefore sensible to ask whether our persistence-type results hold for Mori dream spaces. As in \cite{ablett2024gotzmann}, we say $P(t_1,\dots,t_s)$ is a Hilbert polynomial on the Cox ring $S$ of a Mori dream space $X$ if it is the Hilbert polynomial associated to some homogeneous ideal of $S$. We obtain the following theorem.

\begin{theorem}\label{theorem1}
    Let $X$ be a Mori dream space with Cox ring $S$ and Picard number $s$, and let $J \subseteq S$ be a homogeneous ideal. Suppose that $P(t_1,\dots,t_{s})$ is a Hilbert polynomial on $S$. Then we demonstrate the existence of $2^{s}$ points $(r_1,\dots,r_{s}) \in \mathbb{N}^{s}$ such that checking $H_J(r_1,\dots,r_{s})=P(r_1,\dots,r_{s})$ for all of these points guarantees that $P_J=P$. 
\end{theorem}

Naively, for a subscheme of a $d$-dimensional Mori dream space, the number of non-zero coefficients of the polynomial $P(t_1,\dots,t_s)$ can be as high as $\binom{s+d}{d}$, meaning we would expect to check this many points. In contrast, we find a set of points whose size is not dependent on dimension, which is in keeping with Gotzmann's original persistence result.

In \cite{ablett2024gotzmann} we prove a persistence result for an ideal $I$ in the Cox ring of the product of projective spaces. Given a multigraded Hilbert polynomial $P(t_1,\dots,t_s)$ we show the existence of an $s$-dimensional unit hypercube with the property that verifying $H_I(b_1,\dots,b_s)=P(b_1,\dots,b_s)$ for the vertices of this hypercube ensures that $P_I=P$. The proof of this theorem is constructive, but is reliant on computing the degrees of generators of a multilex ideal in $\cox( \mathbb{P}^{n_1} \times \dots \times \mathbb{P}^{n_s})$ with a given Hilbert polynomial. Ideally we would like to forgo this step and make the computation of the hypercube more straightforward. In the case that the prospective Hilbert polynomial is constant we give an alternative result which is not reliant on computing multilex ideals.

\begin{theorem}\label{theorem2}
Let $I \subseteq S = \cox(\mathbb{P}^{n_1} \times \dots \times \mathbb{P}^{n_s})$ be an ideal, homogeneous with respect to the $\mathbb{Z}^s$-grading. Let $m \in \mathbb{N}$, and suppose $d_i \geq m$ for all $i$, and that $I$ is generated in degrees $\leq(d_1,\dots,d_s) \in \mathbb{Z}^{s}_{\geq 1}$. Let $\mathcal{B} = \{(b_1,\dots,b_s) \in \mathbb{N}^s \mid b_i \in \{d_i,d_i+1\}\}$. 

If \[H_I(b_1,\dots,b_s)=m\] for all $(b_1,\dots,b_s) \in \mathcal{B}$, then $H_I(u_1,\dots,u_s)=m$ for all $(u_1,\dots,u_s) \in (d_1,\dots,d_s) + \mathbb{N}^s$. It follows that $P_I(t_1,\dots,t_s)=m$.
\end{theorem}

In Section \ref{mds}, we show that the Hilbert function of an ideal in the Cox ring of a Mori dream space eventually becomes polynomial and prove Theorem \ref{theorem1}, while in Section \ref{pointpers}, we relate the Hilbert function of a multigraded ideal to that of a module over a standard-graded polynomial ring, leading to Theorem \ref{theorem2}.

\section*{Acknowledgements}
Many thanks to Ali Craw for his suggestion of an extension of my results to Mori dream spaces, and for very helpful mathematical discussions on this topic. Further thanks to Joachim Jelisiejew for interesting discussion about the products of projective spaces case, and to Ritvik Ramkumar who also suggested a potential Mori dream space extension. Many thanks to Diane Maclagan for her much appreciated support throughout the project. The author was funded through the Warwick Mathematics Institute Centre for
Doctoral Training, with support from the University of Warwick and the UK Engineering and Physical Sciences Research Council (EPSRC grant EP/W523793/1).

\section{Persistence for Mori dream spaces}\label{mds}
\subsection{Preliminaries}
Mori dream spaces are a class of varieties known for being well-behaved with respect to birational geometry. We begin with some preliminaries on these spaces and their Cox rings.
\begin{definition}\label{coxring}
    Fix a field $k$ and let $X \subseteq \mathbb{P}^n_k$ be a projective variety with $\pic(X) \otimes_{\mathbb{Z}} \mathbb{Q} \cong N^1(X)$. Let $\tors(X) \subseteq \pic(X)$ denote the torsion subgroup of $\pic(X)$, and let $\eff(X)$ denote the pseudoeffective cone of $X$. Let $[E_1],\dots,[E_{s}]$ be classes of Cartier divisors forming a $\mathbb{Z}$-basis for $\pic(X)/\tors(X)$, and with $\eff(X) \subseteq \sum \mathbb{Q}_{\geq 0}[E_i]$. Then
    \[\cox(X) = \bigoplus_{(e_1,\dots,e_{s}) \in \mathbb{N}^{s}}H^0(X,\mathcal{O}_X(e_1E_1+ \dots + e_{s}E_{s})).\]
    Observe that $\cox(X)$ is graded by $\mathbb{Z}^s \cong \pic(X)/\tors(X)$.
\end{definition}
While this ring appears to depend on our choice of $E_1,\dots,E_{s}$, different choices of $E_1,\dots,E_{s}$ give isomorphic Cox rings \cite{moridreamGIT}*{Section 2}.
\begin{thrmdefinition}[\cite{moridreamGIT}*{Proposition 2.9}]
    Let $X \subseteq \mathbb{P}^n_k$ be a $\mathbb{Q}$-factorial projective normal variety with $\pic(X) \otimes_{\mathbb{Z}} \mathbb{Q} \cong N^1(X)$. Then $X$ is a Mori dream space if and only if $\cox(X)$ is finitely generated as a $k$-algebra.
\end{thrmdefinition}
The key motivation for studying these spaces is that there is an algorithm for the minimal model program which terminates. Note that if $X$ is a Mori dream space then $\nef(X)$ is a rational polyhedral cone. 

Throughout the rest of Section \ref{mds}, let $X \subseteq \mathbb{P}^n_k$ be a Mori dream space with Picard rank $s$, and fix $[E_1],\dots,[E_{s}]$ to be classes of Cartier divisors forming a $\mathbb{Z}$-basis for $\pic(X)/\tors(X)$ with $\eff(X) \subseteq \sum \mathbb{Q}_{\geq 0}[E_i]$. Let $\mathcal{K}$ denote the semigroup of nef line bundles on $X$. Set $S=\cox(X)$ for brevity. Since $S$ is finitely generated there exists a $\mathbb{Z}^{s}$-graded polynomial ring $R=k[x_1,\dots,x_r]$ such that $S = R/I$ for an ideal $I \subseteq R$, homogeneous with respect to the $\mathbb{Z}^{s}$-grading. By \cite{moridreamGIT}*{Proposition 2.11} there is an embedding of $X$ into a simplicial projective toric variety $W$ whose Cox ring is given by $R$, and with $\pic(X) \otimes_{\mathbb{Z}} \mathbb{Q} \cong \pic(W) \otimes_{\mathbb{Z}} \mathbb{Q}$. This embedding will allow us to extend results on simplicial projective toric varieties to Mori dream spaces.
\begin{definition}
The Hilbert function of $S$ is
    \begin{align*}
        H_S \colon \mathbb{N}^{s} & \rightarrow \mathbb{N} \\
        (e_1,\dots,e_{s}) & \mapsto \dimn_k(S_{(e_1,\dots,e_s)}).
    \end{align*}
Observe that $H_S(e_1,\dots,e_s) = h^0(X,\mathcal{O}_X(e_1E_1+\dots+e_{s} E_{s}))$ by definition. For a homogeneous ideal $J \subseteq S$ define the Hilbert function of $J$ to be 
    \begin{align*}
        H_J \colon \mathbb{N}^{s} & \rightarrow \mathbb{N} \\
        (e_1,\dots,e_{s}) & \mapsto \dimn_k((S/J)_{(e_1,\dots,e_s)}).
    \end{align*}
\end{definition}
\subsection{Existence of the Hilbert polynomial}
We claim that the function $H_J$ agrees with a numerical polynomial for Cartier degrees sufficiently far into the nef cone. Note that by numerical polynomial we mean a polynomial with rational coefficients taking non-negative integer values when $t_1,\dots,t_s$ are integers.
\begin{proposition}\label{polyex}
  Let $X$ be a Mori dream space with semigroup of nef line bundles $\mathcal{K}$ and Cox ring $S$. Let $\iota \colon Z \hookrightarrow X$ be the closed subscheme of $X$ defined by the homogeneous ideal $J \subseteq S$. Let $W$ be a simplicial projective toric variety as in \cite{moridreamGIT}*{Proposition 2.11} with Cox ring $R$ such that $\pi \colon R \rightarrow S$ is a surjective map of rings and $\pic(X) \otimes_{\mathbb{Z}} \mathbb{Q} \cong \pic(W) \otimes_{\mathbb{Z}} \mathbb{Q}$. Then there exists a Cartier divisor $D'$ and a numerical polynomial $P_J(t_1,\dots,t_{s}) \in \mathbb{Q}[t_1,\dots,t_{s}]$ such that $H_J(e_1,\dots,e_{s})=P_J(e_1,\dots,e_{s})$ for all $E=e_1E_1+\dots+e_{s}E_{s} \in D' + \mathcal{K}$. Further, this polynomial is the Hilbert polynomial of the ideal $I+\pi^{-1}J \subseteq R$.
\end{proposition}
\begin{proof}
By \cite{MR0726440}*{Theorem (1)}, given an ample Cartier divisor $D$ there exists some $d \in \mathbb{N}$ such that \[H^i(X,\iota_{*}\mathcal{O}_Z(dD+C))=0\] for any nef Cartier divisor $C$ and $i \geq 1$. We conclude that for all $E \in dD + \mathcal{K}$ we have $H^i(X,\iota_{*}\mathcal{O}_Z(E))=0$ for all $i \geq 1$. We then observe that for such $E$, we have \[\chi (\iota_{*}\mathcal{O}_Z(E))=h^0(X,\iota_{*}\mathcal{O}_Z(E)).\] By \cite{snapperpoly}*{Theorem 9.1}, $\chi (\iota_{*}\mathcal{O}_Z(t_1E_1+\dots+t_sE_s))$ is a numerical polynomial in $t_1,\dots,t_{s}$. Further, for \[J' = \bigoplus_{(e_1,\dots,e_s) \in \mathbb{N}^s} H^0(X,\tilde{J}(e_1E_1+\dots+e_sE_s)) \subseteq S\] we have $\tilde{J'}=\tilde{J}$, and by definition $H_{J'}(e_1,\dots,e_s) = h^0(X,\iota_{*} \mathcal{O}_Z(e_1E_1+\dots+e_{s}E_{s}))$ for all $\boldsymbol{e}=(e_1,\dots,e_s) \in \mathbb{N}^s$. Note also by definition that \[H_J(e_1,\dots,e_s) = H_{I+\pi^{-1}J}(e_1,\dots,e_s), \quad H_{J'}(e_1,\dots,e_s)=H_{I+\pi^{-1}J'}(e_1,\dots,e_s)\] where $H_{I+\pi^{-1}J}$ denotes the Hilbert function of $I + \pi^{-1}J \subseteq R$. From a geometric perspective, let $\mu \colon X \hookrightarrow W$ be the embedding of $X$ into $W$ induced by $\pi \colon R \rightarrow S$, and denote $\nu = \mu \circ \iota \colon Z \hookrightarrow W$. We have $h^0(X,\iota_{*} \mathcal{O}_Z(e_1E_1+\dots+e_{s}E_{s}))=h^0(W,\nu_{*}\mathcal{O}_Z(e_1E_1+\dots+e_sE_s)$. By applying \cite{DC1}*{Theorem 3.7} to the ideals $I+\pi^{-1}J,I+\pi^{-1}J' \subseteq R$ we conclude that there exists $F \in \mathcal{K}$ such that \[H_{I+\pi^{-1}J}(e_1,\dots,e_s) = H_{I+\pi^{-1}J'}(e_1,\dots,e_s) = h^0(W,\nu_{*} \mathcal{O}_Z(e_1E_1+\dots+e_{s}E_{s}))\] for all $e_1E_1+\dots+e_sE_s \in F + \mathcal{K}$. Consequently, \[H_J(e_1,\dots,e_s) = H_{J'}(e_1,\dots,e_s) = h^0(X,\iota_{*} \mathcal{O}_Z(e_1E_1+\dots+e_{s}E_{s}))\] for all $e_1E_1+\dots+e_sE_s \in F + \mathcal{K}$. 

Now taking $D' \in (F+\mathcal{K}) \cap (dD + \mathcal{K})$, we have that $H_J(e_1,\dots,e_s)=\chi (\iota_{*}\mathcal{O}_Z(E))$ for $E=e_1E_1+\dots+e_sE_s \in D' + \mathcal{K}$. We therefore set \[P_J(t_1,\dots,t_s) = \chi (\iota_{*}\mathcal{O}_Z(t_1E_1+\dots+t_sE_s))\] and the result follows. Since we had $H_J(e_1,\dots,e_s)=H_{I+\pi^{-1}J}(e_1,\dots,e_s)$ for all $e_1E_1+\dots+e_sE_s \in F + \mathcal{K}$ we conclude that $P_J(t_1,\dots,t_s)=P_{I+\pi^{-1}J}(t_1,\dots,t_s)$. This means that any Hilbert polynomial on $S$ is also a Hilbert polynomial on $R$.
\end{proof}
\begin{remark}
Proposition \ref{polyex} can also be proven algebraically using \cite{MR1357776}*{Theorem 1}.
\end{remark}
\subsection{Persistence}
The goal of this section is to extend the persistence theorem of \cite{ablett2024gotzmann} to the setting of Mori dream spaces. In \cite{ablett2024gotzmann}*{Theorem 1.3} we obtained a persistence theorem for smooth projective toric varieties. In fact, this theorem holds more generally for simplicial projective toric varieties. As in \cite{ablett2024gotzmann}, we apply \cite{maclagan2004multigraded}*{Theorem 1.4} to obtain a persistence theorem for simplicial projective toric varieties.
\begin{theorem}\label{quasismooth}
    Let $W$ be a simplicial projective toric variety with $\pic(W)/\tors(W) \cong \mathbb{Z}^{s}$ and Cox ring $R$. Let $J \subseteq R$ be an ideal, homogeneous with respect to the $\mathbb{Z}^{s}$-grading. Let $P(t_1,\dots,t_s)$ be a Hilbert polynomial on $R$. Then there exists at most $2^{s}$ points $(r_1,\dots,r_{s}) \in \mathbb{N}^{{s}}$ such that checking $H_J(r_1,\dots,r_{s})=P(r_1,\dots,r_{s})$ for all of these points guarantees that $P_J=P$. 
\end{theorem}
\begin{proof}
The proof is identical to that of \cite{ablett2024gotzmann}*{Theorem 1.3}. Let $\reg(W)$ denote the multigraded regularity of $R$, as defined in \cite{maclagan2004multigraded}*{Definition 1.1}. Recall that this is a set contained inside $\pic(W)$. Let $\mathcal{K}$ denote the semigroup of nef line bundles on $W$. Since $W$ is projective, $\nef(W)$ is full dimensional, and so we can fix a basis $\boldsymbol{c_1},\dots,\boldsymbol{c_s} \subseteq \mathcal{K}$ for the vector space $\mathbb{Q}^s$, with $\boldsymbol{c_1} \in \reg(W)$. We further fix $n_i = h^0(W,\mathcal{O}_W(\boldsymbol{c_i}))-1$, and set $T = \cox(\mathbb{P}^{n_1} \times \dots \times \mathbb{P}^{n_s})$. As in \cite{ablett2024gotzmann}*{Definition 4.6}, define 
\begin{align*}
f\colon \mathbb{Q}^{s} &\rightarrow \mathbb{Q}^s \\
(b_1,\dots,b_s) &\mapsto b_1\boldsymbol{c}_1+\dots+b_s\boldsymbol{c}_s.
\end{align*}
Let $f^{\#}$ be the induced map on $\mathbb{Q}[t_1,\dots,t_s] \rightarrow \mathbb{Q}[u_1,\dots,u_s]$. Note that \cite{maclagan2004multigraded}*{Theorem 1.4} holds for a simplicial toric variety as outlined at the beginning of \cite{maclagan2004multigraded}*{Section 6}, meaning \cite{ablett2024gotzmann}*{Lemma 4.9} applies. We conclude that there exists a homogeneous ideal $K \subseteq T$ such that $(T/K)_{\boldsymbol{b}} \cong (R/J)_{f(\boldsymbol{b})}$ for all $\boldsymbol{b}=(b_1,\dots,b_s) \in \mathbb{N}^s$ with $b_1 \geq 1$. It then follows that for a Hilbert polynomial $P(t_1,\dots,t_s)$ on $R$, $f^{\#}P(u_1,\dots,u_s)$ is a Hilbert polynomial on $T$. Since $\nef(W)$ is full dimensional, $f$ is surjective, and consequently $f^{\#}$ is injective. Applying \cite{ablett2024gotzmann}*{Theorem 1.1} and the injectivity of $f^{\#}$ now gives the result.
\end{proof}

We may now extend \cite{ablett2024gotzmann}*{Theorem 1.3} to any Mori dream space.
\begin{proof}[Proof of Theorem \ref{theorem1}]
    Recall that by \cite{moridreamGIT}*{Proposition 2.11}, there is a closed embedding of $X$ into a simpicial projective toric variety $W$ with Cox ring $R$, such that $S=R/I$ for some homogeneous ideal $I \subseteq R$ and $\pic(X) \otimes_{\mathbb{Z}} \mathbb{Q} \cong \pic(W) \otimes_{\mathbb{Z}} \mathbb{Q}$. The result then follows by Proposition \ref{polyex} and Theorem \ref{quasismooth}.
\end{proof}

\begin{exmp}
    Let $X=\gr(2,4) \times \gr(2,4) \subseteq W=\mathbb{P}^5 \times \mathbb{P}^5$. Then \[\cox(X) = k[x_0,\dots,x_5,y_0,\dots,y_5] / (x_0x_5-x_1x_4+x_2x_3,y_0y_5-y_1y_4+y_2y_3),\] with the usual $\mathbb{Z}^2$-grading on $\cox(W)=k[x_0,\dots,x_5,y_0,\dots,y_5]$ given by $\deg(x_i)=(1,0)$, $\deg(y_j)=(0,1)$. Note that $X$ is a Mori dream space by \cite{LW1}*{Lemma 3.3}. Consider the ideal $I = (x_0,x_2,x_3,y_0,y_2,y_3,x_1^2,y_1^2) \subseteq \cox(X)$. Applying Theorem \cite{ablett2024gotzmann}*{Theorem 1.1} to $I' \subseteq \cox(W)$ would allow us to verify whether another ideal in $\cox(X)$ has matching Hilbert polynomial to $I$.
    
    Explicitly, following the language of \cite{ablett2024gotzmann}, observe that the corresponding ideal $I^{\prime} = (x_0,x_2,x_3,y_0,y_2,y_3,x_1^2,x_1x_4,y_1^2,y_1y_4) \subseteq \cox(W)$ is bilex as in \cite{AAKCEDN}*{Definition 4.4} under an appropriate monomial ordering, and that $P_{I^{\prime}}(t_1,t_2) = t_1t_2+2t_1+2t_2+4$. We can find the minimum possible point $(d_1,d_2) \in \mathbb{N}^2$ satisfying the criterion in the proof of \cite{ablett2024gotzmann}*{Theorem 1.1} for $I'$. To do this, note that $I'$ is generated in degrees $\leq(2,2)$, and that 
    \[P_{I'}(t_1,t_2) = (t_1+2)\binom{t_2-2+1}{1} + (3t_1+4)\binom{t_2-2+1}{1}.\] Since the Gotzmann numbers of $t_1+2$ and $3t_1+4$ are $2$ and $7$ respectively, we can fix $d_1=7$. For this value of $d_1$, maximal growth of $P_I(d_1,t_2) = 9t_2+18$ as in \cite{crona2006standard}*{Theorem 4.10} occurs in degree $54$. Thus, the minimum possible point $(d_1,d_2)$ satisfying our criterion for $I'$ is $(7,54)$. This means that by Theorem \ref{theorem1}, for an appropriate ideal $J \subseteq R$ we can verify whether $P_J(t_1,\dots,t_s)=P_I(t_1,\dots,t_s)$ by checking the Hilbert function of $J$ at the points $(7,54), (8,54), (7,55), (8,55)$.
\end{exmp}

\section{Persistence for points in products of projective spaces}\label{pointpers}
In this section we obtain a stronger persistence result for products of projective spaces, in the case the subscheme is zero-dimensional. This is in contrast to \cite{ablett2024gotzmann}*{Theorem 1.1}, which requires knowledge of the generators of certain multilex ideals. We will require the following lemma and definition.
\begin{lemma}[\cite{bruns1998cohen}*{Lemma 4.2.6}]
Fix $d \in \mathbb{Z}_{>0}$ and $\alpha \in \mathbb{N}$. Then $\alpha$ can be written uniquely in the form
\begin{equation*}
    \alpha=\binom{\kappa(d)}{d}+\binom{\kappa(d-1)}{d-1}+\dots+\binom{\kappa(1)}{1},
\end{equation*}
where $\kappa(d)>\kappa(d-1)>\dots>\kappa(1)\geq0$.
\end{lemma}
\begin{definition}
Let $\alpha \in \mathbb{N}$, $d \in \mathbb{Z}_{>0}$. Define
\begin{equation*}
    \alpha^{\left<d\right>}= \binom{\kappa(d)+1}{d+1}+\binom{\kappa(d-1)+1}{d-1+1}+\dots +\binom{\kappa(1)+1}{1+1},
\end{equation*}
with $\kappa(d),\dots,\kappa(1)$ as above. We set $0^{\left<d\right>}=0$.
\end{definition}

Gotzmann's persistence theorem can be extended to finitely generated modules over a standard-graded polynomial ring. The following theorem of Gasharov~\cite{gasharov1997extremal} exhibits Macaulay bound and persistence-type results for such modules.

\begin{theorem}[\cite{gasharov1997extremal}, Theorem 4.2]\label{gash}
    Let $R=k[x_0,\dots,x_n]$, and let \[F=R\xi_1 + \dots + R\xi_v\] be a free, finitely generated $R$-module with $l = \max(\deg(\xi_i) \mid i = 1,\dots,v)$. Let $N \subseteq F$ be a graded submodule and $M=F/N$. Then for any $d \geq l+1$ we have the following:
    \begin{enumerate}[\normalfont(i)]
        \item $\dimn_k M_{d+1} \leq (\dimn_k M_d)^{\left<d-l\right>}$,
        \item if $N$ is generated in degrees $\leq d$ and $\dimn_k M_{d+1} = (\dimn_k M_d)^{\left<d-l\right>}$ then $\dimn_k M_{d+2} = (\dimn_k M_{d+1})^{\left<d-l+1\right>}$.
    \end{enumerate}
\end{theorem}
For the rest of this section let $S=k[x_{1,0},\dots,x_{1,n_1},\dots,x_{s,0},\dots,x_{s,n_s}]$ be the Cox ring of $\mathbb{P}^{n_1} \times \dots \times \mathbb{P}^{n_s}$. 
\begin{lemma}\label{applygash}
    Let $R=k[x_{s,0},\dots,x_{s,n_s}]$ be the standard-graded polynomial ring in $n_s+1$ variables. Let $I \subseteq S$ be a homogeneous ideal, generated in degrees $\leq (a_1, \dots, a_s)$. Then for fixed $d_1, \dots, d_{s-1} \in \mathbb{N}$, we have that $M = \bigoplus_{d_s \in \mathbb{N}} (S/I)_{(d_1,\dots,d_s)}$ is a finitely generated R-module. Specifically, we can write $M \cong F/N$ where $F= R^v$ for some $v \in \mathbb{Z}_{\geq 1}$, and $N \subseteq F$ is a graded $R$-module generated in degrees $\leq a_s$.
\end{lemma}
\begin{proof}
    Let $\xi_1,\dots,\xi_v$ be the monomial basis for $S_{(d_1,\dots,d_{s-1},0)}$. Set $F=R^v$ and observe that $\bigoplus_{d_s \in \mathbb{N}} S_{(d_1,\dots,d_s)} \cong F$ as $R$-modules, with the isomorphism given explicitly by the map \[r_1 \xi_1 + \dots + r_v \xi_v \mapsto (r_1,\dots,r_v) .\] Consider the submodule $\bigoplus_{d_s \in \mathbb{N}} I_{(d_1,\dots,d_s)} \subseteq \bigoplus_{d_s \in \mathbb{N}} S_{(d_1,\dots,d_s)}$, whose image under the isomorphism with $F$ is a graded submodule $N \subseteq F$. We have $M \cong F/N$, and consequently
    \[H_I(d_1,\dots,d_{s-1},u)= \dim_k (F/N)_{u}\] for all $u \geq 0$. It remains to show that $N$ is generated in degrees $\leq a_s$. We will show that $\bigoplus_{d_s=0}^{a_s} I_{(d_1,\dots,d_{s})}$ generates $\bigoplus_{d_s \in \mathbb{N}} I_{(d_1,\dots,d_s)}$ as an $R$-module. Let $I=(f_1,\dots,f_c)$ be a homogeneous generating set for $I$ with $\deg(f_i) \leq (a_1,\dots,a_s)$ for all $i$. Consider an arbitrary element $g=g_1f_1+\dots+g_cf_c \in I_{(d_1,\dots,d_s)}$ for some $d_s \in \mathbb{N}$. For a given $i$ let $\deg(f_i) = (e_{i1},\dots,e_{is})$. We have that $g_i \in S_{(d_1-e_{i1},\dots,d_s-e_{is})}$. We further write $g_i = h_1l_1 + \dots + h_{m_i}l_{m_i}$ for elements $h_j \in S_{(0,\dots,0,d_s-e_{is})}$ and $l_j \in S_{(d_1-e_{i1},\dots,d_{s-1}-e_{is-1},0)}$. Then $l_j f_i \in I_{(d_1,\dots,d_{s-1},e_{is})}$ for each $j$. It follows that for a given $i$ we can write $g_if_i$ as a sum, where each summand is an element in $I_{(d_1,\dots,d_{s-1},e_{is})}$ multiplied by an element of $R$. Since this holds for all $i$, and $e_{is} \leq a_s$ for all $i$, we observe that $\bigoplus_{d_s=0}^{a_s} I_{(d_1,\dots,d_s)}$ generates $\bigoplus_{d_s \in \mathbb{N}} I_{(d_1,\dots,d_s)}$ as an $R$-module. Consequently, the set $\{(r_1,\dots,r_v) \in F \mid r_1\xi_1 + \dots + r_v\xi_v \in \bigoplus_{d_s=0}^{a_s} I_{(d_1,\dots,d_s)} \}$ generates $N$ as an $R$-module. It follows that $N$ is generated in degrees $\leq a_s$.
    \end{proof}

Combining these results allows us to determine when a homogeneous ideal in $S$ has a constant Hilbert polynomial without needing to use results on multilex ideals as in \cite{ablett2024gotzmann}*{Theorem 1.1}.
\begin{proof}[Proof of Theorem \ref{theorem2}]
We proceed by induction. For the base case, fix $(b_1,\dots,b_{s-1}) \in \mathbb{N}^{s-1}$ with $b_i \in \{d_i,d_i+1\}$, and let $v = \dimn_k S_{(b_1,\dots,b_{s-1},0)}$. By Lemma \ref{applygash} there exists $N \subseteq R^v$ generated in degrees $\leq d_s$ such that for $M = R^v/N$ we have \[H_I(b_1,\dots,b_{s-1},u_s)=\dimn_k(M_{u_s})\] for all $u_s \geq 0$. Observe that in the context of Theorem \ref{gash} for $M$ we have $l=0$, and consequently that $d_s  \geq l+1$. Since $d_s \geq m$, $m = m^{\left<d_s\right>}$, and we have 
\begin{equation*}\label{eq1}
    \dimn_k(M_{d_s+1}) = H_I(b_1,\dots,b_{s-1},d_s+1) = m = m^{\left<d_s\right>} = H_I(b_1,\dots,b_{s-1},d_s)^{\left<d_s\right>} = \dimn_k(M_{d_s})^{\left<d_s\right>}.
\end{equation*}
Applying Theorem \ref{gash} we conclude that $\dimn_k(M_{u_s}) =m$ for all $u_s \geq d_s$. Consequently,
\begin{equation*}\label{eq2}
    H_I(b_1,\dots,b_{s-1},u_s)=m 
\end{equation*}
for all $u_s \geq d_s$ and any fixed $(b_1,\dots,b_{s-1}) \in \mathbb{N}^{s-1}$ with $b_i \in \{d_i,d_i+1\}$.

We now assume that 
\[H_I(b_1,\dots,b_{j},u_{j+1},\dots,u_s)=m\]
for some $j \in \{1,\dots,s-1\}$, and for all $(b_1,\dots,b_{j}) \in \mathbb{N}^{j}$ with $b_i \in \{d_i,d_i+1\}$ and $u_i \geq d_i$. We show that this implies
\[H_I(b_1,\dots,b_{j-1},u_{j},\dots,u_s)=m\] for all $(b_1,\dots,b_{j-1}) \in \mathbb{N}^{j-1}$ with $b_i \in \{d_i,d_i+1\}$ and $u_i \geq d_i$.

By assumption, we have 
\[H_I(b_1,\dots,b_{j-1},d_{j}+1,u_{j+1},\dots,u_s) = m = m^{\left<d_{j}\right>} = H_I(b_1,\dots,b_{j-1},d_{j},u_{j+1},\dots,u_s)^{\left<d_{j}\right>}\] for all $(b_1,\dots,b_{j-1}) \in \mathbb{N}^{j-1}$ with $b_i \in \{d_i,d_i+1\}$ and $u_i \geq d_i$. Again applying Lemma \ref{applygash} and Theorem \ref{gash}, now with $v = \dimn_kS_{(b_1,\dots,b_{j-1},0,u_{j+1},\dots,u_{s})}$, observe that 
\[H_I(b_1,\dots,b_{j-1},u_{j},\dots,u_s)=m\] for all $(b_1,\dots,b_{j-1}) \in \mathbb{N}^{j-1}$ with $b_i \in \{d_i,d_i+1\}$, $u_i \geq d_i$. Thus, by induction we obtain
\[H_I(u_1,\dots,u_s)=m\] when $u_i \geq d_i$ for all $i$, and consequently $P_I(t_1,\dots,t_s)=m$.

\end{proof}

\end{document}